\documentclass[12pt,a4paper]{article}

\usepackage[a4paper, total={6in, 10in}]{geometry}
\usepackage{amsmath}
\usepackage{amsthm}
\usepackage{amsfonts}

\usepackage{authblk}

\newtheorem{definition}{Definition}
\newtheorem{theorem}{Theorem}
\newtheorem{corollary}{Corollary}
\newtheorem{lemma}{Lemma}
\newtheorem{proposition}{Proposition}
\newtheorem{remark}{Remark}

\title{Algebraic construction of higher order difference approximations for fractional derivatives and applications}

\author{H. M. Nasir\thanks{nasirh@squ.edu.om}}
\author{K. Nafa\thanks{nkamel@squ.edu.om}}
\affil{Department of Mathematics and Statistics,
Sultan Qaboos University, Sultanate of Oman}

\begin{document}

\maketitle

A generalization of the Gr\"{u}nwald difference approximation for
fractional derivatives  in terms of a real sequence and its generating function is presented.
Properties of the generating function are derived for consistency and
order of accuracy for the approximation corresponding to the generator.
Using this generalization, some higher order Gr\"{u}nwald type approximations
are constructed and tested for numerical stability by using steady state fractional differential problems.
These higher order approximations are used in  Crank-Nicolson type numerical schemes to
approximate the solution of space fractional diffusion equations.
Stability and convergence of these numerical schemes are analysed and are supported by
numerical examples.\\


\section{Introduction}

Fractional calculus has a history that goes back to L'Hospital,
Leibniz and Euler \cite{leibnitz1962letter, eulero1738progressionibus}.
A historical account of early works on factional calculus can be found,
for eg., in \cite{ross1977development}.
Fractional integral and fractional derivative are extensions of the integer
order integrals and derivatives to a real or complex order.
Various definitions of fractional derivatives have been proposed, among which
the Riemann-Liouville, Gr\"{u}nwald-Letnikov and Caputo derivative are common and established.
There are also definitions based on Laplace and Fourier transforms.
Each definition characterizes certain properties of the integer order derivatives.

Recently, fractional calculus found its way into the application domain in
science and engineering. The field of application includes, but not limited to,
oscillation phenomena \cite{mainardi1996fractional}, visco-elasticity \cite{bagley1983theoretical},
control theory \cite{vinagre2000some} and transport problem \cite{metzler2004restaurant}.
Fractional derivative is also found to be more suitable to describe
anomalous transport in an external field
derived from the continuous time random walk \cite{barkai2000continuous},
resulting in a fractional diffusion equation (FDE).
The FDE involves fractional derivative either in time, in space or in both  variables.

A finite difference type approximation for fractional derivative is the Gr\"{u}nwald
approximation obtained from the Gr\"{u}nwald-Letnikov definition.
Numerical experience and theoretical justifications have shown that
application of this approximation {\it as it is}
in the space fractional diffusion equation (SFDE) results in unstable solutions when explicit,
implicit and even the Crank-Nicolson (CN) schemes are used \cite{meerschaert2004finite}.
The latter two schemes are popular for their unconditional stability for
classical diffusion equations with integer order derivatives.
This peculiar phenomenon for the implicit and CN schemes is corrected and the
stability is restored when a shifted
form of the Gr\"{u}nwald approximation
is used \cite{meerschaert2004finite, meerschaert2006finite}.

The Gr\"{u}nwald approximation is known to be of first order with the space
discretization size $h$ in the shifted and non-shifted form and
is, therefore, useful only in first order schemes for the SFDE such as
explicit Euler (forward) and implicit Euler(backward) methods.
Since the CN approximation scheme is of second order in time step $\tau$,
Meerchaert et al. \cite{tadjeran2006second} used extrapolation improvement
for the space discretization to obtain a second order accuracy.
Subsequently, second order approximations for the space fractional
derivatives were obtained through some manipulations on the
first order Gr\"{u}nwald approximation.
Nasir et al. \cite{nasir2013second} obtained a second order accuracy through
a non-integer shift in the Gr\"{u}nwald approximation, displaying super convergence.
Convex combinations
of various shifts of the shifted Gr\"{u}nwald approximation were used to obtain
higher order approximation in Chinese schools \cite{tian2015class,hao2015fourth, zhou2013quasi}.
Zhao and Deng \cite{zhao2015series} extended the concept of super convergence to derive a series of  higher order approximations.

Earlier, Lubich \cite{lubich1986discretized} obtained higher order approximations for the
fractional derivative for orders up to 6 with no shifts involved.
Numerical experiments show that these approximations are also unstable for the
SFDE when the numerical methods mentioned above are used.
Shifted forms of these higher order approximations diminish the order to one,
making them unusable as Chen and Deng \cite{chen2014fourth, chen2014fourthSIAM} observed.
We give a simple proof of this observation from our main result.

The present authors generalized the concept of Gr\"{u}nwald difference approximation
by identifying it with representing generating functions. We establish properties
of this approximating generating function for consistency and order of accuracy. This generalization opens
a door for more choices of Gr\"{u}nwald type approximations.

In this paper, we
construct algebraically some higher order shifted Gr\"{u}nwald type approximations
for the fractional derivative. Interestingly, our newly constructed approximations turn out to be shifted extension counterparts of the Lubich formula.
We apply some of the approximations in the CN type
schemes in FDEs with justification of stability and convergence.

The rest of the paper is organized as follows. In Section \ref{DefinitionSec}, definitions and
notations are introduced. In Section \ref{MainSec}, the main results of generalization is presented. In Section \ref{ConstructSec}, higher order
approximations with shifts are constructed  with some numerical tests to select
suitable approximation methods.
In Section \ref{ApproxSec} some selected methods are applied
to device numerical schemes for space fractional diffusion equation.
Stability and convergence for the schemes
are analysed in Section \ref{StabilitySec}. Supporting
numerical results are presented in Section \ref{NumericalResultSec} and conclusion are drawn in Section \ref{ConclusionSec}.

\section{Definitions and Notations}\label{DefinitionSec}

Let  $L_1(\Omega) = \left\{ f | \int_{-\infty}^{\infty} |f(x)|dx < \infty \right\} $ denote
the space of Lebesgue integrable functions.
The Fourier transform (FT) of $f(x) \in L_1(\mathbb{R})$ is defined by
$
\mathfrak{F}(f(x))(\eta) =: \hat{f}(\eta) = \int_{-\infty}^{\infty} f(x)e^{-i\eta x}dx., \eta \in \mathbb{R} $.
The inverse FT is
$
\mathfrak{F}^{-1}({\hat f}(\eta))(x) = \int_{-\infty}^{\infty} {\hat f} (\eta)e^{i\eta x}d\eta = f(x).
$
The FT is linear: for $\alpha, \beta \in \mathbb{R}$, and $ f,g \in L_1(\mathbb{R}) $,
$
\mathfrak{F}(\alpha f(x)+\beta g(x))(\eta) = \alpha \hat{f}(\eta)+ \beta {\hat g}(\eta).
$
For a function $f$ at a point $x+\beta \in \mathbb{R}$,   the FT is given by
$
\mathfrak{F}(f(x+\beta))(\eta) = e^{i\beta x} \hat{f}(\eta).
$\\
For $\alpha \in \mathbb{R}$,  denote $ {\cal C}^{m+\alpha}(\mathbb{R}) =
\left\{ f \in L_1(\mathbb{R}) | \int_{-\infty}^{\infty}
(1+|\eta|)^{\alpha} {\hat f}(\eta) d\eta \right\} $.\\

The left and right Gr\"{u}nwald-Letnikov (GL)fractional derivatives are given respectively by

\begin{equation}\label{LeftGL}
  \;_{-\infty} D_x^\alpha f(x) = \lim_{h \rightarrow 0} \frac{1}{h^\alpha } \sum_{k=0}^\infty (-1)^k \binom{\alpha}{k} f(x-kh)
\end{equation}

\begin{equation}\label{RightGL}
  \;_x D_{\infty}^\alpha f(x) = \lim_{h \rightarrow 0} \frac{1}{h^\alpha } \sum_{k=0}^\infty (-1)^k \binom{\alpha}{k} f(x+kh)
\end{equation}

where $\binom{\alpha}{k} = \frac{\Gamma (\alpha +1)}{\Gamma (\alpha +1-k)k!} $.


The FT of the left and right fractional derivatives are
$\mathfrak{F}(\;_{-\infty}D_x^\alpha f(x) ) (\eta) =
(i\eta)^\alpha \hat{f}(x) $ and
$\mathfrak{F}(\;_xD_\infty^\alpha f(x) ) (\eta) =
(-i\eta)^\alpha \hat{f}(x) $ respectively \cite{podlubny1998fractional}.

For a fixed $h$, the finite difference type Gr\"{u}nwald approximations for the
left and right fractional derivatives
are obtained by simply dropping the limit in the GL definitions (\ref{LeftGL}) and
(\ref{RightGL}) as
\begin{equation}\label{LeftGrunwaldApprox}
  \delta_{-h}^\alpha f(x) = \frac{1}{h^\alpha } \sum_{k=0}^\infty g_k^{(\alpha)} f(x-kh), \qquad
  \delta_{+h}^\alpha f(x) = \frac{1}{h^\alpha } \sum_{k=0}^\infty g_k^{(\alpha)} f(x+kh),
\end{equation}
where
$g_k^{(\alpha)} = (-1)^k \binom{\alpha}{k}$ are the coefficients of the
Taylor series expansion of the generating function $(1-z)^\alpha $.
These coefficients can be computed by the recursive formula
$
 g_0^{(\alpha)}=1, \quad g_k^{(\alpha)}=\left( 1-\frac{\alpha+1}{k} \right) g_{k-1}^{(\alpha)} ,\quad k = 1,2,3,\cdots ,
$ and  satisfy the properties
$
g_1^{(\alpha)} = -\alpha \le 0 ,\hspace{5mm}
\sum_{k=0}^\infty g_k^{(\alpha)} = 0 ,\hspace{5mm}
\sum_{k=0}^M g_k^{(\alpha)}\le 0, \hspace{5mm} \forall M\ge 2.
$\\

When $f(x) $ is defined in the intervals $[a,b]$, it is zero-extended outside the interval to adopt these definitions of fractional derivatives and their approximations. The
sums are restricted to  finite up to $N$
which grows to infinity as $h \rightarrow 0$.
Often, $N$ is chosen to be $ N = \left[ \frac{x-a}{h} \right] $ and
$ N = \left[ \frac{b-x}{h} \right] $ for the left and right fractional derivatives  respectively to cover the sum up to the boundary of these domain intervals, where $[y]$ is the integer part of $y$. The left and right fractional derivatives, in this case, are denoted by $ \;_aD^\alpha_x $ and
$ \;_xD^\alpha_b $ respectively.

The Gr\"{u}nwald approximations are  of first order accuracy and display unstable solutions
in the approximation of SFDE by implicit and CN type schemes \cite{meerschaert2004finite}. As a remedy,
shifted forms of left and right Gr\"{u}nwald formulas with  shift $r$ are used:
\begin{equation}\label{LshiftedGApp}
    \delta_{\mp h,r}^\alpha f(x) = \frac{1}{h^\alpha } \sum_{k=0}^{N+r} g_k^{(\alpha)} f(x \mp (k-r)h), \quad  x > a,
\end{equation}

where
$N=\left[ \frac{x-a}{h} \right]$ or $N =
\left[ \frac{b-x}{h} \right]$ for left and right fractional derivatives respectively and
the upper limits of the summations have been adjusted to cover the shift $r$.

Meerchaert et al. \cite{meerschaert2004finite} showed that
for a shift $r = 1$ , the $\delta_{\mp h,1}^\alpha f(x)$
are again of first order approximations with unconditional stability restored
in implicit and CN type schemes for SFDEs.

For higher order approximations, Nasir et al. \cite{nasir2013second} derived a second order approximation by a non-integer shift, displaying super convergence.
\begin{equation}\label{Nasir2nd}
\delta_{-h,\alpha/2} f(x) = \;_aD_x^\alpha f(x) + O(h^2).
\end{equation}

Chen et al.\cite{tian2015class} used convex combinations
of different shifted Gr\"{u}nwald forms to obtain two order 2 approximations.
\begin{equation}\label{Tian2nd}
  \lambda_1 \delta_{-h,p} +\lambda_2 \delta_{-h,q}
  = \;_aD_x^\alpha f(x) + O(h^2), \qquad \lambda_1 +\lambda_2 = 1
\end{equation}
with $\lambda_1 = \frac{\alpha - 2q}{2(p-q)} $ and
$\lambda_2 = \frac{ 2p-\alpha }{2(p-q)} $  for
$(p,q) = (1,0) , (1,-1)$.
In the above, we stated the form for left fractional derivative only and analogous forms for right FD hold.

Hao et al. \cite{hao2015fourth} obtained a quasi-compact order 4 approximation
by incorporating the second order term in the fractional derivative.
All the above approximations were derived from the fundemental GL approximation with coefficients $g^{(\alpha)}_{k}$ of the generating function $(1-z)^\alpha$.

Higher order approximations were also obtained earlier by Lubich
\cite{lubich1986discretized} establishing a connection with the
characteristic polynomials of multistep methods for ordinary differential
equations.
Specifically,
if $ \rho(z), \sigma(z) $ are, the
characteristic polynomials of a multistep method of convergence order $p$ \cite{henrici1962discrete},
then $\left(\frac{\sigma(1/z)}{\rho(1/z)}\right)^\alpha $ gives the coefficients for the
Gr\"{u}nwald type approximation of same order for the fractional derivative of order $\alpha$.
From the backward multistep methods, Lubich \cite{lubich1986discretized} derived
approximations of orders up to order six in the form
$ L_p (z) = (\sum_{j=1}^p \frac{1}{j} (1-z)^j )^\alpha $.
The generating functions
$L_p(z)$ given in Table \ref{LubichHighOrder} are of order $p$ for $ 1\le p \le 6 $ for the non-shift Gr\"{u}nwald type approximations.

\begin{table}[h]
  \centering
  \begin{tabular}{ll}
    \hline
 $L_1(z) = \left(1-z\right)^\alpha   $ \\
 $L_2(z) = \left(\frac{3}{2} -2z+\frac{1}{2}z^2\right)^\alpha $ \\
 $L_3(z) = \left(
\frac{11}{6} -3z+\frac{3}{2}z^2-\frac{1}{3}z^3
\right)^\alpha  $ \\
 $L_4(z) = \left(
\frac{25}{12} -4z+3z^2-\frac{4}{3}z^3 + \frac{1}{4} z^4 \right)^\alpha $ \\
 $L_5(z) = \left(
\frac{137}{60} -5z+5z^2-\frac{10}{3}z^3 + \frac{5}{4} z^4 -\frac{1}{5}z^5
\right)^\alpha $ \\
 $L_6(z) = \left(
\frac{49}{20} -6z+\frac{15}{2}z^2-\frac{20}{3}z^3 + \frac{15}{4} z^4 -\frac{6}{5}z^5 +\frac{1}{6} z^6
\right)^\alpha $ \\
    \hline
  \end{tabular}
  \caption{Lubich approximation generating functions}\label{LubichHighOrder}
\end{table}

As noted for the non-shifted Gr\"{u}nwald approximation, these higher order
approximations also display unstable solutions with implicit Euler and CN type schemes. Moreover, The shifted form of these approximations suffer the orders dropping to one and hence uninteresting.

\section{Higher order shifted approximations}\label{MainSec}
In this section, we present the generalization of the Gr\"{u}nwald approximation established in \cite{nasir2017AnewSecond} with important results are given with proof for completion.

For a function $f(x)$,
denote the left and right Gr\"{u}nwald type operator  with shift $r$ and weights $w_{k,r}^{(\alpha)}$,  respectively, as
\begin{equation}\label{GrunwaldOp}
\Delta_{\mp h,r}^\alpha f(x) = \frac{1}{h^\alpha } \sum_{k=0}^\infty
w_{k,r}^{(\alpha)} f(x \mp (k-r) h).
\end{equation}

\begin{definition}\label{def1}
A sequence $ \left\{ w_{k,r}^{(\alpha)} \right\} $ of real numbers is said to approximate the fractional derivatives
$\;_{-\infty} D_x^\alpha  f(x) $ and
 $\;_x D_\infty^\alpha  f(x) $  at $x$ with shift $r$ in the sense of Gr\"{u}nwald if
\begin{equation}\nonumber\label{consistency1}
\;_{-\infty} D_x^\alpha  f(x) =
\lim_{h \rightarrow 0} \Delta_{-h,r}^\alpha f(x),
\qquad
\;_x D_{\infty}^\alpha  f(x) =
\lim_{h \rightarrow 0} \Delta_{+h,r}^\alpha f(x).
\end{equation}

\end{definition}

\begin{definition}\label{def2}
A sequence $ \left\{ w_{k,r}^{(\alpha)} \right\} $ of real numbers is said
to approximate the fractional derivatives $\;_{-\infty} D_x^\alpha  f(x) $ and
 $\;_x D_\infty^\alpha  f(x) $
with shift $r$  and order $p \ge 1$ if
\begin{equation}
\;_{-\infty} D_x^\alpha  f(x) =
\Delta_{-h,r}^\alpha f(x) + O(h^p),
\qquad
\;_x D_\infty^\alpha  f(x) =
\Delta_{+h,r}^\alpha f(x) + O(h^p).
\end{equation}
\end{definition}

We denote the generating function of the coefficients
$w_{k,r}^{(\alpha)} $ as
\[
W(z) = \sum_{k=0}^{\infty} w_{k,r}^{(\alpha)} z^k.
\]

An equivalent characterization
of the generator $W(z)$ for an approximation of fractional differential operator
with order  $ p\ge 1$ and shift $r$ is given by

\begin{theorem}\label{Maintheorem}
Let $n-1 < \alpha \le n, \, m$ be a non-negative integer, $f(x)\in C^{m + n + 1}(\mathbb{R})$ and $ \;_\infty D_x^k f(x), \;_x D_\infty^k f(x) \in L^1(\mathbb{R})$ for $ 0 \le k \le m+n+1$. Then,
the generating function $W(z)$ of a real sequence $\{ w_{k,r}^{(\alpha)}\}$ approximates the left and right fractional differential
operators for $f(x)$ with  order $p$ and shift $r$, $1\le p\le m $, if and only if
\begin{equation}\label{Worderp}
   G_r(z) := \frac{1}{z^\alpha} W(e^{-z})e^{rz} = 1 + O(z^p).
\end{equation}
Moreover, if $ G_r(z) = 1 + \sum_{l=p}^{\infty} a_{l}(r) z^l $, then
we have for the left fractional derivative
\begin{equation}\label{TaylorFractional}
\Delta_{-h,r}^\alpha f(x) =
\;_{-\infty} D_x^\alpha f(x) + h^p a_p(r) \;_{-\infty} D_x^{p+\alpha} f(x)+\cdots + O(h^m),
\end{equation}
\begin{equation}\label{TaylorFractionalR}
\Delta_{+h,r}^\alpha f(x) =
\;_x D_\infty^\alpha f(x) + h^p a_p(r) \;_x D_\infty^{p+\alpha} f(x)+\cdots + O(h^m)
\end{equation}
\end{theorem}

\begin{proof}
We prove the result for the left fractional derivative.
Taking FT of $\Delta_{h,+r}^\alpha f(x) $ in (\ref{GrunwaldOp}),
with the help of linearity, we obtain
\begin{align*}
\mathfrak{F} (\Delta_{h,+r}^\alpha  f(x))(\eta)  & = \frac{1}{h^\alpha } \sum_{k=0}^\infty
w_{k,r}^{(\alpha)} \mathfrak{F}(f(x-(k-r) h))(\eta) \\
   & = \frac{1}{h^\alpha } \sum_{k=0}^\infty w_{k,r}^{(\alpha)}e^{-(k-r)ih\eta} \hat{f}(\eta)\\
   & = \frac{e^{rih\eta}}{(ih\eta)^\alpha }  \sum_{k=0}^\infty w_{k,r}^{(\alpha)} e^{-kih\eta}(i\eta)^\alpha \hat{f}(\eta)\\
   & = \frac{e^{rz}}{z^\alpha}  \left[  \sum_{k=0}^\infty w_{k,r}^{(\alpha)} (e^{-z})^k  \right]
   (i\eta)^\alpha \hat{f}(\eta)  \\
   & = \frac{e^{rz}}{z^\alpha} W(e^{-z}) (i\eta)^\alpha \hat{f}(\eta) = G_r(z) (i\eta)^\alpha \hat{f}(\eta)\\
   & = \sum_{l=0}^\infty a_{l}(r) z^l (i \eta)^\alpha {\hat f}(\eta)
    = \sum_{l=0}^\infty a_{l}(r) h^l (i \eta)^{l+\alpha} {\hat f}(\eta),
   \end{align*}
where we have used $z = i\eta h $.
Applying inverse FT, we have
\begin{equation}\label{MainApprox}
\Delta_{-h,r}^\alpha  f(x)
= \sum_{l=0}^{m-1} a_{l}(r)  \;_a D_x^{l+\alpha}  f(x) h^l +O(h^m).
\end{equation}
Now, (\ref{Worderp}) holds if and only if $ a_{0}(r) = 1, a_l(r) = 0$, for $ l = 1,2,\cdots, p-1$.
Moreover, (\ref{TaylorFractional}) holds from (\ref{MainApprox}) with (\ref{Worderp}).
\end{proof}

One of the consequence of Theorem \ref{Maintheorem} is the following consistency condition.

\begin{corollary}\label{Cor1}
If the generating function $W(z)$ gives a consistent approximation of the left and right fractional differential operator, then the following hold:
\begin{enumerate}
  \item $W(1)=0$,
  \item $\sum_{k=0}^\infty w_{k,r}^{(\alpha)} = 0$,
  \item The fractional derivative of constant is zero.
\end{enumerate}
\end{corollary}

\begin{proof}
1. Since the order $p$ is at least one, when $h\ne 0$, and hence $z\ne 0$,  the condition (\ref{Worderp}) becomes
  $W(e^{-z})e^{\pm rz} = z^\alpha ( 1 + O(z^p))$. Take limit as $h \rightarrow 0  (z \rightarrow 0) $.

  2. follows immediately with (\ref{Worderp}) and
  3. If $f(x) = C$, then (\ref{GrunwaldOp}) gives
  $\Delta_{\mp h,r} C = C \sum_{k=0}^{\infty} w^{(\alpha)}_{k,r} = 0$ and so are their limits as $h\rightarrow 0$.
\end{proof}

Using Theorem \ref{Maintheorem}, one can easily check algebraically the following propositions for the generating function of approximation operators known previously in
\cite{nasir2013second, tian2015class, lubich1986discretized, chen2014fourth}.

\begin{proposition}
The approximation given by (\ref{Nasir2nd}) is of second order accuracy.
\end{proposition}
\begin{proof}
The coefficients $g_k^{(\alpha)}$ have the generating function $(1-z)^\alpha$. Since the shift of the approximation is $r = \alpha/2 $, it is enough to check the function
$ G(z) = \frac{1}{z^\alpha}e^{z\alpha/2 }  (1-e^{-z})^\alpha $. Taylor series expansion gives that
$G(z) = 1 + \alpha^2/24 z^2+ O(z^4)$ which confirms the second order.
\end{proof}

\begin{proposition}
The approximation given by (\ref{Tian2nd}) is of second order accuracy.
\end{proposition}
\begin{proof}
Since $\delta_{-h,p}$ and $\delta_{-h,q}$  have the same generating function $(1-z)^\alpha $ with shifts $p$ and $q$ respectively, The Taylor series of
$G(z) =  \frac{\lambda_1}{z^\alpha}e^{pz}  (1-e^{-z})^\alpha
+ \frac{\lambda_2}{z^\alpha}e^{qz}  (1-e^{-z})^\alpha $ which gives $G(z) = 1 + (- \frac{\alpha^{2}}{8} + \frac{\alpha p}{4} + \frac{\alpha q}{4} + \frac{\alpha}{24} - \frac{p q}{2})z^2+ O(z^3) $. The coefficient of $z^2$ is non-zero for all integer shifts $p$ and $q$.
\end{proof}

\begin{proposition}
The Lubich generating functions $W_p(z), p = 1,2,\cdots,6,$ in Table \ref{LubichHighOrder} are
of order $p$ accurate without shift. Moreover, if a non-zero shift $r$ is introduced  to the approximation, the order reduces to one.
\end{proposition}
\begin{proof}
When there is no shift ($r = 0$), Taylor series expansion gives $W_p(e^{-z})/z^\alpha = 1 + O(h^p)$.
When a shift $r \ne 0$  is introduced to $W_p(z)$, the order is determined by
\[
e^{rz}W_p(e^{-z})/z^\alpha = (1 + O(z))(1+O(z^p)) = 1+O(z), \hspace{3mm} for  \;\; 1\le p \le 6
\]
reducing the order to 1.
\end{proof}

\section{Construction of higher order approximations}\label{ConstructSec}

We construct generating functions for higher order approximation with shifts by the use of
Theorem \ref{Maintheorem} and its. The importance of Theorem \ref{Maintheorem} is that the construction process is
entirely confined to algebraic manipulation with the aid of Taylor series expansion.

Theorem \ref{Maintheorem} opens a door to choose a variety of forms for generating functions for
$W(z)$. In this paper, we choose the form
\begin{equation}\label{GenWmr}
 W(z) = (\beta_0 + \beta_1 z + \beta_2 z^2 +\cdots + \beta_p z^p )^\alpha
\end{equation}
to construct generators.

We have the following consistency
condition for
the coefficients $u_k $,  $k\ge 0$.

\begin{theorem}\label{Cor2}
If the generating function $W(z) $ in (\ref{GenWmr}) approximates the fractional derivative, then
\begin{equation*}\label{consistency}
 \beta_0 + \beta_1  + \beta_2 +\cdots + \beta_p = 0
\end{equation*}
\end{theorem}

\begin{proof}
Using Corollary \ref{Cor1}(1).
\end{proof}

Generating function of the approximation of order $p$ with shift $r$ can be
algebraically constructed by the following algorithm.

\begin{enumerate}
  \item Set $ W(z) = (\beta_0 + \beta_1 z + \beta_2 z^2 +\cdots + \beta_p z^p )^\alpha $, where the coefficients $\beta_i$ are to be determined.

  \item Expand
  $G(z) = \frac{W(e^{-z})e^{rz}}{z^\alpha} $ in Taylor series:\\
   \begin{equation}\label{GTaylor}
   G(z) = a_0 + a_1 z + \cdots +a_l z^l +\cdots ,
   \end{equation}
   where $a_l \equiv a_l(\beta, r,\alpha)$.
  \item Form the system of equations by imposing order conditions\\
  $a_0 = 1, a_l = 0  $ for $l = 1, 2,3,\cdots p-1$.
  \item Solve the systems for $\beta_i, \quad i = 0, 1,\cdots,  p$ with the
    additional consistency condition $\beta_0 +\beta_1 +\cdots+\beta_p = 0$.
\end{enumerate}

We have obtained approximating generating functions $W_{p,r}(z)$ with order $p$, shift $r$ for $1\le p\le 6$
listed in Table \ref{Wtable}.

\begin{table}[h]
\centering
\begin{tabular}{ll}
\hline\hline
$p$ & $\beta_k, 0\le k \le p$\\
\hline\hline
1 & $ \beta_0 = 1$ \\
& $ \beta_1 = -1$\\
\hline
2 & $ \beta_0   = \frac{3}{2} - \frac{r}{\alpha}  $ \\
& $ \beta_1   = -2 + \frac{2 r}{\alpha} $ \\
& $ \beta_2   =\frac{1}{2} - \frac{r}{\alpha} $\\
\hline
3 & $ \beta_0    = \frac{11}{6} - \frac{2 r}{\alpha} + \frac{r^{2}}{2 \alpha^{2}} $ \\
& $ \beta_1    =  -3 + \frac{5 r}{\alpha} - \frac{3 r^{2}}{2 \alpha^{2}} $  \\
& $ \beta_2    =  \frac{3}{2} - \frac{4 r}{\alpha} + \frac{3 r^{2}}{2 \alpha^{2}} $\\
& $ \beta_3    =  - \frac{1}{3} + \frac{r}{\alpha} - \frac{r^{2}}{2 \alpha^{2}} $ \\
\hline
4 &$ \beta_0    = \frac{25}{12} - \frac{35 r}{12 \alpha} + \frac{5 r^{2}}{4 \alpha^{2}} - \frac{r^{3}}{6 \alpha^{3}} $\\
& $ \beta_1    =  -4 + \frac{26 r}{3 \alpha} - \frac{9 r^{2}}{2 \alpha^{2}} + \frac{2 r^{3}}{3 \alpha^{3}} $\\
& $ \beta_2    =  3 - \frac{19 r}{2 \alpha} + \frac{6 r^{2}}{\alpha^{2}} - \frac{r^{3}}{\alpha^{3}} $ \\
& $ \beta_3    =  - \frac{4}{3} + \frac{14 r}{3 \alpha} - \frac{7 r^{2}}{2 \alpha^{2}} + \frac{2 r^{3}}{3 \alpha^{3}} $\\
& $ \beta_4    =  \frac{1}{4} - \frac{11 r}{12 \alpha} + \frac{3 r^{2}}{4 \alpha^{2}} - \frac{r^{3}}{6 \alpha^{3}} $ \\
\hline
5 & $ \beta_0    = \frac{137}{60} - \frac{15 r}{4 \alpha} + \frac{17 r^{2}}{8 \alpha^{2}} - \frac{r^{3}}{2 \alpha^{3}} + \frac{r^{4}}{24 \alpha^{4}} $ \\
& $ \beta_1    =  -5 + \frac{77 r}{6 \alpha} - \frac{71 r^{2}}{8 \alpha^{2}} + \frac{7 r^{3}}{3 \alpha^{3}} - \frac{5 r^{4}}{24 \alpha^{4}} $\\
& $ \beta_2    =  5 - \frac{107 r}{6 \alpha} + \frac{59 r^{2}}{4 \alpha^{2}} - \frac{13 r^{3}}{3 \alpha^{3}} + \frac{5 r^{4}}{12 \alpha^{4}} $\\
& $ \beta_3    =  - \frac{10}{3} + \frac{13 r}{\alpha} - \frac{49 r^{2}}{4 \alpha^{2}} + \frac{4 r^{3}}{\alpha^{3}} - \frac{5 r^{4}}{12 \alpha^{4}} $\\
& $ \beta_4    =  \frac{5}{4} - \frac{61 r}{12 \alpha} + \frac{41 r^{2}}{8 \alpha^{2}} - \frac{11 r^{3}}{6 \alpha^{3}} + \frac{5 r^{4}}{24 \alpha^{4}} $ \\
& $ \beta_5    =  - \frac{1}{5} + \frac{5 r}{6 \alpha} - \frac{7 r^{2}}{8 \alpha^{2}} + \frac{r^{3}}{3 \alpha^{3}} - \frac{r^{4}}{24 \alpha^{4}} $ \\
\hline
6 & $ \beta_0    = \frac{49}{20} - \frac{203 r}{45 \alpha} + \frac{49 r^{2}}{16 \alpha^{2}} - \frac{35 r^{3}}{36 \alpha^{3}} + \frac{7 r^{4}}{48 \alpha^{4}} - \frac{r^{5}}{120 \alpha^{5}} $ \\
& $ \beta_1    =  -6 + \frac{87 r}{5 \alpha} - \frac{29 r^{2}}{2 \alpha^{2}} + \frac{31 r^{3}}{6 \alpha^{3}} - \frac{5 r^{4}}{6 \alpha^{4}} + \frac{r^{5}}{20 \alpha^{5}} $ \\
& $ \beta_2    =  \frac{15}{2} - \frac{117 r}{4 \alpha} + \frac{461 r^{2}}{16 \alpha^{2}} - \frac{137 r^{3}}{12 \alpha^{3}} + \frac{95 r^{4}}{48 \alpha^{4}} - \frac{r^{5}}{8 \alpha^{5}} $\\
& $ \beta_3    =  - \frac{20}{3} + \frac{254 r}{9 \alpha} - \frac{31 r^{2}}{\alpha^{2}} + \frac{121 r^{3}}{9 \alpha^{3}} - \frac{5 r^{4}}{2 \alpha^{4}} + \frac{r^{5}}{6 \alpha^{5}} $\\
& $ \beta_4    =  \frac{15}{4} - \frac{33 r}{2 \alpha} + \frac{307 r^{2}}{16 \alpha^{2}} - \frac{107 r^{3}}{12 \alpha^{3}} + \frac{85 r^{4}}{48 \alpha^{4}} - \frac{r^{5}}{8 \alpha^{5}} $ \\
& $ \beta_5    =  - \frac{6}{5} + \frac{27 r}{5 \alpha} - \frac{13 r^{2}}{2 \alpha^{2}} + \frac{19 r^{3}}{6 \alpha^{3}} - \frac{2 r^{4}}{3 \alpha^{4}} + \frac{r^{5}}{20 \alpha^{5}} $\\
& $ \beta_6    =  \frac{1}{6} - \frac{137 r}{180 \alpha} + \frac{15 r^{2}}{16 \alpha^{2}} - \frac{17 r^{3}}{36 \alpha^{3}} + \frac{5 r^{4}}{48 \alpha^{4}} - \frac{r^{5}}{120 \alpha^{5}} $ \\
\hline
\end{tabular}
\caption{Coefficients $\beta_k$ of
$W_{p,r}(z) = (\beta_0 + \beta_1 z + \cdots + \beta_p z^p  )^\alpha $} for order $p$, shift $r$ approximation, $1\le p\le 6$.\label{Wtable}
\end{table}

Note that when there is no shift, i.e., $r = 0$, these generating functions
reduces to those in Table \ref{LubichHighOrder} obtained by Lubich \cite{lubich1986discretized}.

The Gr\"{u}nwald weights $w_{k,r}^{(\alpha)} $ can be computed by a recurrence formula
\cite{rall1981automatic, weilbeer2005efficient}.

It seems that, despite
 higher order of approximation of $W_{p,r}(z)$, the stability of solutions
using these approximations to FDEs remains an issue. Our numerical tests
on some steady state problems show that the second order approximation
$W_{2,1}(z) $ displays stability for values of $\alpha $ in the interval $[1,2]$.
However, for the other  higher order approximations,
the stability is limited to a subset $[\alpha_0, 2], 1<\alpha_0, $   of the
interval $[1,2]$. This phenomenon warrants an investigation
for approximations of orders three and above.

From now on, we focus our attention on the second order approximation $W_{2,r}(z) $.
We denote the approximation operators corresponding to the generators $W_{p,r}(z)$  by $\Delta_{p,\pm r}$ for left and right fractional differential operators.

We have the following for $W_{2,r}(z) $ from Theorem \ref{Maintheorem}.

The second order approximations of generator $W_{2,r}(z) $ is given by
\begin{equation}\label{Order2}
  \Delta_{2,+r}^\alpha u(x) = \;_{-\infty} D_x^\alpha u(x) + O(h^2),
\end{equation}

\begin{equation}\label{Order2R}
  \Delta_{2,-r}^\alpha u(x) = \;_x D_\infty^\alpha u(x) + O(h^2),
\end{equation}

Another approach to construct higher order approximation generators $W(z)$ is to retain the non-zero coefficient $a_2(r)$
in (\ref{GTaylor}) of step 2 of the algorithm without imposing the vanishing condition. That is, one imposes conditions $a_0(r) = 1, a_l(r) = 0$ for $ l = 1, 3,4, \cdots, p$ for order $p$. This technique is similar to the method used in 
\cite{zhou2013quasi}.

Then we have
\[
G_r(z) = 1 + a_2(r) z^2 + O(z^{p}).
\]
Equation (\ref{MainApprox})  then gives
\begin{align*}\label{QCAm}
\Delta_{p,+r}^\alpha  u(x) & = \;_{-\infty} D_x^\alpha u(x) + a_2(r) h^2 \;_{-\infty} D_x^{2+\alpha} u(x) + O(h^{p})\nonumber\\
& = \;_{-\infty} D_x^\alpha u(x) + a_2(r) h^2 D^2 \;_{-\infty} D_x^{\alpha} u(x) + O(h^{p})\nonumber\\
& = P_x \;_{-\infty} D_x^\alpha u(x) + O(h^{p}),
\end{align*}
where
\begin{equation*}\label{Px}
P_x = \left(I + h^2 a_2(r) D^2 \right),
\end{equation*}
with $D^2 = d^2/dx^2 $  and the identity operator $I$.

The differential operator  $D^2$ is approximated by the second order central difference operator $\delta_h^2 $ given by
\[
\delta_h^2 u(x) = \frac{1}{h}( u(x-h) - 2u(x) + u(x+h) )
\]
with
\[
    \delta_h^2 u(x) = D^2u(x) + O(h^2).
\]
From this, an approximation for the operator $P_x$ is obtained as
\begin{equation}\label{PhOperator}
    P_h u(x) = ( 1 + a_2(r) h^2 \delta_h^2 )u(x)
\end{equation}
with order 4 accuracy,
\[
    P_h u(x) = P_xu(x) + O(h^4).
\]

Hao et al. \cite{hao2015fourth} derived a fourth order approximation from
the first order Gr\"{u}nwald approximation (\ref{LshiftedGApp})  by considering a convex combination of three of its shifted forms with an appropriate $P_x$ and called it a quasi-compact approximation. YanYan Yu et al. \cite{yu2017third} used this technique to derive third order schemes
for tempered fractional diffusion equations. In those papers, the approximation schemes were derived directly from the first order Gr\"{u}nwald approximation $(1-z)^\alpha$.

We use the properties of general generating functions of approximations to derive higher order schemes as it allows various choices with only algebraic manipulations.

We obtained a third order quasi-compact approximation  from the second order approximation generator $W_{2,r}(z)$ without additional vanishing conditions as
\begin{equation}\label{QCA3}
  \Delta_{2,\pm r}^\alpha u(x) = P_x \;_{-\infty} D_x^\alpha u(x)  + O(h^3),
\end{equation}
with
\begin{equation*}\label{CompactCoeff3}
     a_2(r) = - \frac{\alpha}{3} + r - \frac{r^{2}}{2 \alpha}.
\end{equation*}

\subsection{ Numerical tests }\label{NumericalTestSec}

We test the approximation operators of orders 2 and 3 derived in this section applying them to the steady state problem
\begin{align}\label{steadystate}
  \;_aD_x^\alpha u(x) & = f(x),  \quad a\le x \le b, \\
  u(a) &= \phi_0,  \quad u(b) = \phi_1. \nonumber
\end{align}
We briefly describe the approximation schemes for the steady state problem.\\

For a uniform partition $ a=x_0 < x_1 < x_2 < \cdots < x_N = b$ of the domain $[a,b]$ with subinterval size $h$, problem (\ref{steadystate}) at $x_i$ is approximated by the operator $\Delta_{2,+1}$. Then one gets from (\ref{Order2}),
\[
\Delta_{2,+1} u_i = f_i + O(h^2), \quad i = 0,1,2,\cdots N,
\]
where  $u_i = u(x_i)$ and $f_i = f(x_i), i = 0,1,2,\cdots,N-1$.
Neglecting the $O(h^2)$ remainder term, we get the approximation scheme
\begin{equation}\label{steadyApprox2}
\Delta_{h,+1} {\hat u}_i = f_i , \quad i = 0, 1,2,\cdots N,
\end{equation}
where ${\hat u}_i $ the approximation of the exact solution $u_i$ for $i = 0, 1,2,\cdots, N-1$ with $\hat{u}_0 = \phi_0,   \hat{u}_N = \phi_1.$

Let ${ U} = [u_0, {\hat u}_1,{\hat u}_2,\cdots,{\hat u}_{N-1},u_N]^T$,
$ F = [ f_0, f_1,f_2,\cdots, f_N]^T$ with the
boundary conditions incorporated in ${ U}$. Then, the matrix formulation of (\ref{steadyApprox2}) is given by
$
A_{2,1} U = F,
$
where $A_{2,1}$ is an $(N+1) \times (N+1)$ toeplitz matrix given by
\begin{equation*}\label{AalphaMatrix}
A_{2,1} (i,j) = \left\{
\begin{array}{cc}
  w_{i-j + 1,1}^{(\alpha)}, & i \ge j-1 \\
  0, & elsewhere.
\end{array}
\right.
\end{equation*}

Bringing the boundary values to the right side, the reduced system becomes,
\begin{equation}\label{Reduced2}
{\hat A}_{2,1} {\hat U} = {\hat F} - A_0\phi_0 - A_N \phi_1    ,
\end{equation}
where ${\hat A}_{2,1}$ is the reduced matrix obtained from $A_{2,1} $ by deleting the first and last rows and columns, ${\hat U}, {\hat F} $ are obtained from $U,F$ respectively by deleting their first and last boundary entries. $A_0,A_N$
are the first and last column vectors of the matrix $A_{2,1}$ reduced at both ends as above.
The approximate solution is then obtained by solving the matrix equation (\ref{Reduced2}).

For the right fractional derivative $\;_x D_\infty u(x) $, the approximation
matrix is given by $A_{2,1}^T $.

We derive a third order quasi-compact approximation as follows.
We  pre-multiply  (\ref{steadystate}) by the  operator $P_x$:
\begin{equation*}\label{Conditioned}
   P_x\;_aD_x^\alpha u(x)  = P_x f(x),  \quad a\le x \le b.
\end{equation*}

From (\ref{QCA3}), the operator $P_x \;_aD_x^\alpha$ is approximated by $\Delta_{2,+1}$ with order 3. At the grid point  $x_i$, we get
\[
\Delta_{2,+1} u_i = P_h f_i + O(h^3), \quad i = 0,1,2,\cdots N,
\]
where $P_h$ is given in (\ref{PhOperator}).

It is expressed in matrix form as
$
A_{2,1} U = P F,
$
where $P$ is the matrix corresponding to $P_h$ and given by the tri-diagonal matrix
\begin{equation}\label{Tridiagonal}
  P = Tri[ a_2(1), 1 - 2h^2 a_2(1) ,a_2(1) ].
\end{equation}

After  imposing boundary conditions, we obtain the ready-to-solve third order scheme as
\begin{equation}\label{QCapprox}
{\hat A}_{2,1} {\hat U} = {\hat P} {\hat F} - A_0\phi_0 - A_N \phi_1   ,
\end{equation}
where ${\hat P} $ is the reduced matrix of $P$.

Note that we have the same approximation operator $\Delta_{2,+1}$ for
both approximations of orders 2 and 3 and hence have the same matrix $A_{2,1}$.
The pre-multiplication of $P_x$ improves the order of accuracy from 2 to 3. We prefer to call the operator $P_x$ a preconditioner
to the fractional differential operator.

We test the approximation scheme devised in this subsection using the steady state test problem
\begin{align*}\label{SteadytestExample}
  \;_0 D_x^\alpha u(x) & = \frac{10\Gamma{(n+1)}}
  {\Gamma{(n+1-\alpha)} } x^{n-\alpha}
  , \quad 0\le x\le 1, \\
  u(0) & = 0 , \quad
  u(1)  = 10 \nonumber
\end{align*}
with the exact solution  $ u (x) = 10 x^n $. We set $ n = 8 $ and test for various values
of the parameter $\alpha$.

First, we test the second order approximation (\ref{Reduced2}) with $\alpha = 1.1, 1.5$ and $1.9$.
The number of grid subintervals $N$ corresponding to the discretization size
$h=(1-0)/N $ was considered for values $ N=16,32,\cdots,1024$.
The maximum error
$\|u - U\|_\infty $ and the computed convergence orders
 are listed in Table \ref{Order2SteadyTable}.\\

\begin{table}[h]
  \centering\footnotesize
  \begin{tabular}{rllllll}
  \hline
  &  $\alpha = 1.1 $ & &  $\alpha = 1.5 $ & &  $\alpha = 1.9 $\\
  \hline
    $N$ & $\|u - U\|_\infty $ & Order & $\|u - U\|_\infty $ & Order
    & $\|u - U\|_\infty $ & Order\\
  \hline
  16  & 4.8893e-01  &  ----  & 2.5141e-01  &  ---- & 1.3365e-01  &  ----  \\
  32  & 1.1592e-01  &  2.08  & 6.4851e-02  &  1.95 & 3.3951e-02  &  1.98  \\
  64  & 2.7227e-02  &  2.09  & 1.6450e-02  &  1.98 & 8.5491e-03  &  1.99 \\
 128  & 6.3685e-03  &  2.10  & 4.1396e-03  &  1.99 & 2.1446e-03  &  2.00  \\
 256  & 1.4873e-03  &  2.10  & 1.0383e-03  &  2.00 & 5.3703e-04  &  2.00  \\
 512  & 3.5020e-04  &  2.09  & 2.5997e-04  &  2.00 & 1.3437e-04  &  2.00 \\
1024  & 8.7574e-05  &  2.00  & 6.5044e-05  &  2.00 & 3.3606e-05  &  2.00 \\
\hline
  \end{tabular}
  \caption{Second order Approximation with shift $r = 1$ using $W_{2,r}$}\label{Order2SteadyTable}
\end{table}

Next, we test the order 3 quasi-compact approximation (\ref{QCapprox}) with the same setting of parameters used for the previous testing. The test results are given in Table \ref{QCA3SteadyTable}.
\begin{table}[h]
  \centering\footnotesize
  \begin{tabular}{rllllll}
  \hline
  &  $\alpha = 1.1 $ & &  $\alpha = 1.5 $ & &  $\alpha = 1.9 $\\
  \hline
    $N$ & $\|u - U\|_\infty $ & Order & $\|u - U\|_\infty $ & Order
    & $\|u - U\|_\infty $ & Order\\
  \hline
  16  & 9.8696e-03  &  ----  & 1.3027e-02  &  ---- & 3.8208e-03  &  ---- \\
  32  & 1.0719e-03  &  3.15  & 1.6435e-03  &  3.00 & 4.6147e-04  &  3.03\\
  64  & 1.2038e-04  &  3.13  & 2.0611e-04  &  3.00 & 5.6560e-05  &  3.01 \\
 128  & 1.3765e-05  &  3.11  & 2.5805e-05  &  3.00 & 7.0003e-06  &  3.01 \\
 256  & 1.5891e-06  &  3.11  & 3.2281e-06  &  3.00 & 8.7069e-07  &  3.00 \\
 512  & 1.8439e-07  &  3.01  & 4.0366e-07  &  3.00 & 1.0857e-07  &  3.00 \\
1024  & 2.2872e-08  &  3.00  & 5.0467e-08  &  3.00 & 1.3563e-08  &  2.96 \\
\hline
  \end{tabular}
  \caption{Third order Approximation with shift $r = 1$ using $W_{2,1}$}\label{QCA3SteadyTable}
\end{table}
These tests confirm the theoretical justifications of the
orders of the two schemes.

\section{Approximation of fractional diffusion equation}\label{ApproxSec}
We apply the approximations constructed in the previous section to the numerical approximation of the space fractional diffusion
equation defined in the domain $[a,b]\times[0,T]$:
\begin{equation}\label{diffusioneq}
  \frac{\partial u(x,t)}{\partial t} =
  K_1 \;_aD_{x}^{\alpha}u(x,t) +K_2 \;_xD_{b}^{\alpha}u(x,t) + f(x,t),
\end{equation}
with the initial and boundary conditions
\begin{align*}
u(x,0) &= s_0(x) , & x\in[a,b]\\
u(a,t) &= \phi_1(t), u(b,t) = \phi_2(t), &t\in[0,T],\label{initboun}
\end{align*}
where $u(x.t)$ is the unknown function to be determined; $K_1, K_2$ are non-negative constant
diffusion coefficients with $K_1 + K_2 \ne 0$, i.e, not both are simultaneously zero, and
$f(x,t)$ is a known source term.
The boundary conditions are set as follows:
If $K_1 \ne 0$, then $\phi_1(t)\equiv 0$ and
if $K_2 \ne 0$, then $\phi_2(t)\equiv 0$.
We assume that the diffusion problem has a unique solution.

The space domain $[a,b]$ is partitioned into a uniform mesh of size $N$
with subintervals of length $h = (b-a)/N$, and the time domain $[0,T]$ into a uniform partition of size $M$
with subintervals of length $\tau = T/M$. These two partitions form a uniform partition of
the 2-D domain $[a,b]\times[0,T]$ with grid points $(x_i,t_m)$,
where $x_i = a + ih$  and $ t_m = m\tau,  \quad 0\le i \le N ,\quad  0\le m\le M$.
We use the following notations for conciseness: $u_i^m = u(x_i,t_m)$ ,
$t_{m+1/2} = \frac{1}{2}(t_{m+1}+t_m)$ and $ f_i^{m+1/2} = f(x_i, t_{m+1/2})$.

We present the CN type scheme with the third order approximation in (\ref{QCA3}) for the space fractional derivative using $W_{2,1}(z)$ with the preconditioner operator
$P_x = 1 + h^2 a_2(1) D^2 $.  For the second order approximation in (\ref{Order2}), $P_x$ will be the unit operator $I$.

Preconditioning (\ref{diffusioneq}) by $P_x$, one gets the
equivalent equation
\begin{equation*}\label{ModiDiffusion}
P_x \delta_t  u(x,t) =
{\bf D} u(x,t)
+ P_x f(x,t),
\end{equation*}
where ${\bf D} = K_1 P_x \;_aD_{x}^{\alpha} + K_2 P_x  \;_xD_{b}^{\alpha} $.\\

Using the approximations $\Delta_2 = K_1 \Delta_{2,+1}^\alpha  + K_2 \Delta_{2,-1}^\alpha$ of order 3 for $D$,   $P_h$ of order 4 for $ P_x$ and the second order approximations\\
$
   \frac{\partial u(x,t)}{\partial t} = \frac{u(x,t+\tau)-u(x,t)}{\tau} +O(\tau^2) $ and
$ u(x,t+\tau/2) = \frac{u(x,t+\tau)+u(x,t)}{2}+O(\tau^2)$ ,
the Crank-Nicolson type scheme at
$(x_i, t_m), 0\le i \le N-1, 0\le m \le M$, is given by
\begin{equation}\label{ApproxDiffusion1}
P_h \frac{u_i^{m+1}-u_i^m}{\tau} =
\Delta_{2}\frac{1}{2}(u_i^{m+1}+u_i^m)
+ P_h f_i^{m+1/2} +O(\tau^2 +h^p), \quad p = 2,3.
\end{equation}
Here, $P_h = I $ for $p = 2$.

Let ${U}^m $ be the solution of (\ref{ApproxDiffusion1})
after neglecting the $O(\tau^2 + h^p)$ terms with
${U}^m = [\hat{u}_0^m, {\hat u}_1^m, {\hat u}_2^m,\cdots,{\hat u}_{N-1}^m, \hat{u}_N^m]^T $,where  ${\hat u}_i^n $
becomes the approximation of the exact values $u_i^n$.
Then, (\ref{ApproxDiffusion1}) becomes

\begin{equation*}\label{CNtypeApprox}
P_h ({{\hat u}_i^{m+1}-{\hat u}_i^m}) =
\frac{\tau}{2}\Delta_{2}({\hat u}_i^{m+1}+{\hat u}_i^m)
+ \tau P_h f_i^{m+1/2}, 0\le i \le N, 0 \le m \le M-1.
\end{equation*}

Thus, the Crank-Nicolson type scheme in matrix form reads
\begin{equation}\label{CN1}
 P (U^{m+1} - U^m) =
B (U^{m+1}+U^m) + \tau P F^{m+1/2}, 0\le m\le M-1 ,
\end{equation}
   where $ F^{m+1/2} = \tau [f_0^{m+1/2},f_2^{m+1/2},\cdots, f_{N}^{m+1/2}]^T $,
$P $ is given in (\ref{Tridiagonal}).
The matrix $B $ is corresponding to the operator $\Delta_{2}$ given by
$
B = \frac{\tau}{2} (
  K_1 A_{2,1}  +
K_2 A_{2,1}^T).
$
Re-arranging for $U^{m+1} $ and $U^m$,
we have
\begin{equation}\label{CNsystem}
  (P-B) U^{m+1} = (P+B)U^{m} + \tau P F^{m+1/2}, \;  m = 0,1,2,\cdots, M-1.
\end{equation}

 Let ${\hat P} $ and
$ {\hat B} $ be the reduced matrix from $P $ and  $B$ respectively, and
${\hat F}^{m+1/2} $ be the reduced vector from $F^{m+1/2}$ as was in Section \ref{NumericalTestSec}.

After imposing the boundary conditions, equation
(\ref{CNsystem}) reduces to the ready-to-solve form
\begin{equation*}\label{CNsystem1}
  ({\hat P}-{\hat B}){\hat U}^{m+1} = ({\hat P}+{\hat B}){\hat U}^{m} + \tau {\hat P}{\hat F}^{m+1/2} + {\bf \hat b}^m, \;  m = 0,1,2,\cdots, M-1,
\end{equation*}
where
$ {\bf \hat b}^m = {B}_0 (u^{m+1}_0 + u^{m}_0) + {B}_N (u^{m+1}_N + u^{m}_N)    $ and
$ {B}_0 , {B}_N $ are the first($0^{th}$) and last($N^{th}$)
column vectors of the matrix $B$ reduced again as before.

\section{Stability and convergence}\label{StabilitySec}

We establish the stability of the Crank-Nicolson scheme (\ref{CNsystem}) for the approximations of orders 2 and 3  given in (\ref{Order2}) and (\ref{QCA3}).

We closely follow the analysis in \cite{hao2015fourth} and
present some required results.

Let $V_h = \{ v | v= (v_0, v_1 , \cdots, v_N), v_i \in \mathbb{R} , v_0 = v_N = 0 \} $ be the space of grid functions in the computational domain  in space interval $[a, b]$.

For any $u,v \in V_h$, define the discrete inner products and corresponding norms as
\begin{align*}\label{innerproducts}
  (u,v) & = h \sum_{i=1}^{N-1} u_i v_i, & \langle \delta_h u, \delta_h v\rangle  &= h \sum_{i=1}^{N-1} (\delta_h u_{i-1/2})(\delta_h v_{i-1/2}),\\
  \|u\| &= \sqrt{(u,u)},   &  |u|_1 &= \sqrt{\langle \delta_h u, \delta_h u\rangle}.
\end{align*}

Then, we have the following:
\begin{lemma}\leavevmode\label{lem2}
\begin{enumerate}
  \item The operator $P_h$ is self-adjoint.
  \item The quadratic form $(P_h u, u)$ satisfies
  $\frac{1}{5}\|u\| < (P_h u,u) \le \|u\| $ and hence
the norm $\|u\|_P = (P_h u,u) $ is equivalent to $\|u\|$.
\end{enumerate}
\end{lemma}

\begin{proof}
1. $ (P_h u,v) = ( (1+h^2 a_2(1) \delta_h^2)u,v)
= (u,v) + h^2 a_2(1) (\delta_h^2 u,v)  =
(u,v) - h^2 a_2(1) (\delta_h u,\delta_h v)
 =(u,v) + h^2 a_2(1) (u,\delta_h^2 v) = (u,P_h v)$. \\

2. For $1 \le \alpha \le 2 $, we easily see that
$ \frac{1}{12} \le a_2(1) \le 1 - \frac{\sqrt{6}}{3} < \frac{1}{5}$.
Also,\\ $|u|_1^2 = (\delta_h u,\delta_h u)=-(\delta_h^2 u,u) $ and \\
 $  |u|_1^2 \le \| \delta_h^2 u\|  \|u\|\le
\|\frac{1}{h^2}\sum_{i=1}^{N-1} (u_{i+1}-2u_i+u_{i-1} )\| \|u\|
\le \frac{4}{h^2} \|u\|^2. $
\\
Now,
$ (P_h u,u) = (u,u) + h^2 a_2(1) (\delta_h^2 u,u) = \|u\|^2 -h^2 a_2(1) |u|_1^2 \le \|u\|^2$.\\
Also,
$ (P_h u,u) =  \|u\|^2 - h^2 a_2(1) |u|_1^2 > \|u\|^2 -\frac{4}{5}  \|u\|^2 = \frac{1}{5}\|u\|^2 $.
Thus, the equivalence of the norm follows.
\end{proof}

\begin{lemma}\label{ToeplitzNegative}
Let $\{t_{\pm k}\}_{k=0}^\infty $ be a double sided real sequence such that
(i) $t_k + t_{-k} \ge 0$ for $k\ne 0 $ ,  (ii) $\sum_{j=-N}^N t_j \le 0 $ for $ N \ge 0 $.
Then, the toeplitz matrices  $T_N = [t_{i-j}] $ of size $N+1$ are negative definite for $ N\ge 0$.
\end{lemma}

\begin{proof}
For $N=0$, $t_0 \le 0$. This matrix of size 1 is negative definite.
For any positive integer $N$ and for any $(N+1)$-dimensional non-zero vector
${\bf v} = [ v_0,v_1, v_2, \cdots,v_N]^T $,
consider the quadratic form ${\bf v}^T T {\bf v} = \sum_{i=0}^N \sum_{j=0}^N t_{j-i}v_i v_j$.
Summing the terms diagonally, we have
\begin{align*}
{\bf v}^T T_N {\bf v}
&= \sum_{k=-N}^{N} t_k \sum_{j=0}^{N-1-k} v_j v_{k+j}\\
&= t_0 \sum_{j=0}^{N} v_j^2 + \sum_{k = 1}^N (t_k+t_{-k})  \sum_{j=0}^{N-1-k} v_j v_{k+j}\\
&\le t_0 \|{\bf v}\|^2 + \sum_{k = 1}^N (t_k+t_{-k})  \sum_{j=0}^{N-1-k} (1/2)(|v_j|^2 +  |v_{k+j}|^2)\\
& \le t_0 \|{\bf v}\|^2 + \sum_{k = 1}^N (t_k+t_{-k})  \|{\bf v}\|^2
= \sum_{k = -N}^N t_k  \|{\bf v}\|^2 \le 0
\end{align*}
\end{proof}

\begin{lemma}\label{NegDef2_3}
The operators $\Delta_{2,+1}^\alpha $ and $\Delta_{2,-1}^\alpha $ are self-adjoint
and negative definite. Moreover, the
matrices $A_{2,1}$ and $A_{2,1}^T $ corresponding to the operators $\Delta_{2,+1}^\alpha $ and $\Delta_{2,-1}^\alpha $ respectively are negative definite.
\end{lemma}

\begin{proof}
By virtue of Lemma \ref{ToeplitzNegative}, it is enough to prove that
the coefficients $w_{k,1} $ of the  generator
$W_{2,1}(z)= \sum_{k=0}^\infty w_{k,1} z^k $
 satisfy the following properties for $1\le \alpha \le 2 $.
\[
  w_{0,1} \ge 0, \quad  w_{0,1}+w_{2,1} \ge 0 \quad \text{and}\quad
    \sum_{k=0}^M w_{k,1} \le 0  \text{ \rm for all  } M\ge 2.
\]
Let $W_{2,1}(z) = (\beta_0 + u_1 z + \beta_2 z^2)^\alpha$.
Then, \\$\beta_0= \frac{3}{2}-\frac{1}{\alpha}\ge 0 $, \quad
$ \beta_1 = -{2}-\frac{2}{\alpha} \le 0$ \quad and \quad
$\beta_2 = \frac{1}{2}-\frac{1}{\alpha} \le 0$ for $1\le \alpha \le 2$.\\
Also, we have
$w_{0,1}  = \beta_0^\alpha \ge 0$ and  $ w_{1,1} = \alpha \beta_0^{\alpha-1}\beta_1 \le 0$.
Moreover,
\begin{align*}
  w_{0,1} + w_{2,1} & = \beta_0^\alpha + \frac{1}{2 } \alpha \beta_{0}^{\alpha-2}\left[(\alpha-1) \beta_{1}^{2} + 2 \beta_{0} \beta_{2} \right]\\
   & = \frac{1}{2}\beta_0^{\alpha-2} \left[ \alpha(\alpha-1) \beta_1^2 + 2 \beta_0(\alpha \beta_2 +\beta_0 ) \right]\\
   & = \frac{1}{2}\beta_0^{\alpha-2} \left[ \alpha(\alpha-1) \beta_1^2
   + 2 \beta_0\frac{(\alpha-1)(\alpha+2)}{2\alpha} \right] \ge 0,\\
w_{3,1}  &=
\frac{ 1 }{6 }\beta_{0}^{\alpha-3}  \alpha \left(\alpha - 1\right)\beta_{1} \left[(\alpha -2) \beta_{1}^{2} + 6 \beta_{0} \beta_{2} \right] \ge 0
\end{align*}
We show that $ w_m \ge 0$ for all $m\ge 4 $ inductively.
\begin{align*}
  w_{4,1} & = \frac{ 1}{4!} \beta_{0}^{\alpha-4}\alpha\left(\alpha - 1\right) \left[
(\alpha-2)(\alpha-3) \beta_{1}^{4} + 12 (\alpha-2) \beta_{0} \beta_{1}^{2} \beta_{2}  + 12 \beta_{0}^{2} \beta_{2}^{2} \right]  \ge 0 \\
  w_{5,1} & = \frac{\alpha(\alpha-1)(\alpha-2) \beta_0^{\alpha-5}\beta_1}{5!}
  \left[ (\alpha-3)( (\alpha-4)\beta_1^4  + 20 \beta_0 \beta_1^2 \beta_2) + 60\beta_0^2 \beta_2^2\right] \ge 0
\end{align*}

Assuming inductively that $w_{k,1}$ are non negative for $3\le k \le m-1$, we have for $m\ge 6$,\\
$ w_{m,1} = \frac{1}{m\beta_0} ( (\alpha+1-m) w_{m-1,1} \beta_1
+ (2\alpha+2-m)w_{m-2,1}\beta_2) \ge 0$ for $1\le \alpha\le 2$,
since $w_{m-1,1} , w_{m-2,1}$,
$ (\alpha+1-m), (2\alpha+2-m), \beta_1 $ and $\beta_2 $ are all non-positive for $m\ge 6 $.

From Corollary \ref{Cor1}, we have the consistency condition $\sum_{k=0}^\infty w_k = 0 $ in
general which is true for $W_{2,1}(z)$ as well.
Since $ \sum_{k=m}^\infty w_{k,1} \ge 0 $ for $m\ge 3$, the last inequality  follows from the consistency condition.
\end{proof}

\begin{lemma}\label{lemma5}
The approximation operator $\Delta_2 $ is negative definite.
\end{lemma}
\begin{proof}
For any $v \in V_h$ , since the diffusion coefficients
$K_1, K_2 $ are non-negative,
we have\\
$ (\Delta_2 v, v) =
 K_1 (\Delta_{2,+1}^\alpha v, v)  + K_2 (\Delta_{2,-1}^\alpha v,v) \le 0 $.

\end{proof}

\begin{theorem}\label{Estimate}
Let $ v^m = [v_1^m,v_2^m,\cdots, v_{N-1}^m]$ be the solution of the problem
\begin{align}
P_h\delta_t v_i^{m+1/2} &- \Delta_{2} v_i^{m+1/2} = S_i^m, \quad 1\le i \le N-1 , \quad 0\le m\le M-1, \label{Modeleq}\\
v_0^m & = 0, \quad v_M^m = 0, \nonumber\\
v_i^0 & = v_0(x_i) , \quad 0 \le i \le N.\nonumber
\end{align}
Then,
\[
\|v^m\| \le \sqrt{5}\left( \|v^0\| + \sqrt{5}\tau \sum_{l=0}^{m-1} \|S^l\| \right),
\]
where $ S^m = [S_1^m, S_2^m, \cdots, S_{N-1}^m]$.
\end{theorem}

\begin{proof}
Taking inner product  of (\ref{Modeleq}) with $v^{m+1/2}$, we have
\begin{equation*}
(P_h\delta_t v^{m+1/2}, v^{m+1/2}) - (\Delta_{2} v^{m+1/2}, v^{m+1/2}) = (S^m,v^{m+1/2}).
\end{equation*}
Since $ - (\Delta_{2} v^{m+1/2}, v^{m+1/2}) \ge 0$ from Lemma \ref{lemma5},
we have
\begin{align}
(P_h\delta_\tau v^{m+1/2}, v^{m+1/2}) & \le (S^m,v_i^{m+1/2}) = (S^m, v^{m+1/2}) \nonumber \\
& \le \|S^m\| \|v^{m+1/2}\|  \le \sqrt{5} \|S^m\| \|v^{m+1/2}\|_P \nonumber\\
& \le \frac{\sqrt{5}}{2} \|S^m\| \left(\|v^{m+1}\|_P +\|v^{m}\|_P\right).\label{ineq1}
\end{align}
Since $\delta_t v^{m+1/2} = \frac{1}{\tau} (v^{m+1} - v^m )$ and
$ v^{m+1/2} = \frac{1}{2} (v^{m+1} + v^m)$,
we have
\begin{align}
(P_h\delta_t v^{m+1/2}, v^{m+1/2})
& = \left(P_h \frac{1}{\tau} (v^{m+1} - v^m) , \frac{1}{2} (v^{m+1} + v^m)\right) \nonumber\\
& = \frac{1}{2\tau}\left( (P_h v^{m+1} , v^{m+1}) - (P_h v^{m} , v^{m})  \right) \nonumber\\
&= \frac{1}{2\tau} ( \|v^{m+1}\|_P^2 - \|v^{m}\|_P^2 )\label{ineq}
\end{align}
Combining (\ref{ineq1}) and (\ref{ineq}), we have
\[
\|v^{m+1}\|_P \le \|v^{m}\|_P + \sqrt{5} \tau \|S^m\|, \quad 0\le m\le M-1.
\]
Recursively applying this from $m$, we have
\[
\|v^m\|_P \le \|v^0\|_P + \sqrt{5} \tau
\sum_{l=0}^{m-1 }\|S^l\| ,\quad 0\le m\le M-1.
\]
From the equivalence of the two norms by lemma \ref{lem2}, we conclude
\[ \|v^m\|
\le \sqrt{5} \left(\|v^0\| + \sqrt{5} \tau
\sum_{l=0}^{m-1 }\|S^l\| \right), \quad 0\le m\le M-1.
\]
\end{proof}

\begin{remark}
In the above lemma, if $P_h $  is the unit operator $I$, for the case of the order 2 with CN type scheme, since the
fractional approximation operator $\Delta_h^\alpha $ is the same, the above estimate reduces to
\[ \|v^m\|
\le \|v^0\| + \tau
\sum_{l=0}^{m-1 }\|S^l\| , \quad 0\le m\le M-1.
\]
\end{remark}

From the above estimates, we have the following stability result.
\begin{theorem}
  The CN type difference schemes (\ref{CNtypeApprox}) of orders 2 and 3 are unconditionally stable for $1\le \alpha \le 2$.
\end{theorem}

For the convergence of the approximate solution from the
CN type schemes, we have the following.
\begin{theorem}
The approximate solutions of the CN type scheme (\ref{CNtypeApprox}) with the given initial and boundary conditions are convergent as $h,\tau \rightarrow 0$ for $1 \le \alpha \le 2$.
\end{theorem}
\begin{proof}
Let $ e^m = u^m - {\hat U}^m $ be the error vector of the solutions,  where $u^m, {\hat U}^m$ are the exact and approximate solutions of the diffusion problem (\ref{diffusioneq}).
Then the error of the internal grid values ${\hat e}^m $ satisfy the system
\begin{align*}
P_h\delta_t e_i^{m+1/2} &- \Delta_{2} e_i^{m+1/2} = R_i^m, \quad 1\le i \le N-1 , \quad 0\le m\le M-1\\ \label{Modeleq}
e_0^m & = 0, \quad e_M^m = 0\\
e_i^0 & = 0 , \quad 0 \le i \le N.
\end{align*}

Theorem \ref{Estimate} gives the estimate
\[
\| e^m \| \le 5\tau \sum_{l = 0}^{m-1}\|R^l\| \le 5c\tau N (\tau^2 + h^p), \quad p = 2,3.
\]
The convergence is then established as $h,\tau \rightarrow 0$.
\end{proof}

\section{Numerical Results}\label{NumericalResultSec}

We consider the following test example for the fractional
diffusion problem (\ref{diffusioneq}).\\
Let $G(x, m,\alpha)  = \frac{\Gamma(m+1)}{\Gamma(n+1-\alpha)}
(x^{m-\alpha}  + (1-x)^{m-\alpha} ) $ and $ s_0(x) = x^5(1-x)^5 $.
\begin{align*}
  \text{Diffusion coefficients:} & \quad  K_1  = 1,   K_2 = 1, \\
  \text{Source function:} &\quad f(x,t) =  -e^{-t}(s_0(x) + G(x, 5,\alpha ) -5G(x, 6,\alpha),\\
  &+ 10G(x, 7,\alpha) - 10G(x, 8,\alpha) + 5G(x, 9,\alpha)- G(x, 10,\alpha) ), \\
  \text{Initial condition:} &\quad u(x,0)  = s_0(x),  \\
  \text{Boundary conditions:} &\quad u(0,t) = 0, \quad  u(1,t)  = 0, \\
  \text{Exact Solution:} &\quad   u(x,t) =  s_0(x) e^{-t},\\
  \text{Space domain:} &\quad  0\le x \le 1,\\
  \text{Time domain:} &\quad  0\le t \le 1,
\end{align*}

\begin{table}[h]
  \centering\footnotesize
  \begin{tabular}{rllllll}
  \hline
  &  $\alpha = 1.1 $ & &  $\alpha = 1.5 $ & &  $\alpha = 1.9 $\\
  \hline
    $N=1/h$ & $\|u^n - U^n\|_\infty $ & Order & $\|u^n - U^n\|_\infty $ & Order
    & $\|u^n - U^n\|_\infty $ & Order\\
  \hline
  16 &  1.0544e-05  &  ---- & 9.0719e-06  &  ---- & 5.6905e-06  &  ----  \\
  32 &  2.8172e-06  &  1.90 & 2.3208e-06  &  1.97 & 1.4309e-06  &  1.99  \\
  64 &  7.3008e-07  &  1.95 & 5.8863e-07  &  1.98 & 3.5731e-07  &  2.00  \\
 128 &  1.8606e-07  &  1.97 & 1.4836e-07  &  1.99 & 8.9332e-08  &  2.00  \\
 256 &  4.6984e-08  &  1.99 & 3.7252e-08  &  1.99 & 2.2338e-08  &  2.00  \\
 512 &  1.1806e-08  &  1.99 & 9.3341e-09  &  2.00 & 5.5852e-09  &  2.00   \\
  \hline
  \end{tabular}
  \caption{Second order of convergence for CN type scheme with $W_{2,1}$} \label{DiffusionTest1}
\end{table}

The test problem was applied to the CN type numerical schemes developed in Section
\ref{ApproxSec}.

The second order approximation $W_{2,1}(z)$ was tested for $\alpha = 1.1,1.5$ and $1.9$. Table \ref{DiffusionTest1} lists the maximum error and the order of convergence for grid sizes $N = M =  16,32, \cdots , 512 $. The partition subinterval sizes are then
$\tau = 1/M $ and $ h = 1/N $ for time and space, respectively.

For the order 3 approximation, we choose $ \tau = h^{3/2}$ to numerically realize the order 3 of the approximation operator. This means that
for a space partition size $N$, we choose the time partition size $M \approx N^{3/2}$.
In Table \ref{QCA3Eg1}, the test results with the space and time partition sizes $N$ and $M$ chosen are displayed.

\begin{table}[h]
  \centering\footnotesize
  \begin{tabular}{rrllllll}
  \hline
    & $\alpha = 1.1 $ & &  $\alpha = 1.5 $ & &  $\alpha = 1.9 $\\
  \hline
    $N=1/h$ & $M=1/\tau$ & $\|u^n - U^n\|_\infty $ & Order & $\|u^n - U^n\|_\infty $ & Order
    & $\|u^n - U^n\|_\infty $ & Order\\
  \hline
  16 &   65 & 1.9461e-06  &  -  & 7.2807e-07  &  - & 2.9010e-08  &  - \\
  32 &  182 & 2.4807e-07  &  2.97  & 9.1351e-08  &  2.99 & 2.7484e-09  &  3.40 \\
  64 &  513 & 3.1332e-08  &  2.99  & 1.1401e-08  &  3.00 & 5.3796e-10  &  2.35 \\
 128 & 1449 & 3.9404e-09  &  2.99  & 1.4224e-09  &  3.00 & 7.9399e-11  &  2.76 \\
 256 & 4097 & 4.9422e-10  &  3.00  & 1.7758e-10  &  3.00 & 1.0667e-11  &  2.90  \\
 512 & 11586 & 6.1888e-11  &  3.00 & 2.2183e-11  &  3.00 & 1.3792e-12  &  2.95  \\
  \hline
  \end{tabular}
  \caption{Order 3 convergence of CN type scheme }\label{QCA3Eg1}
\end{table}

These test results show that the second order approximation and
the third order approximation operator with $W_{2,1}$  are justified for their order of convergence and unconditional stability with the CN type schemes.

\section{Conclusion}\label{ConclusionSec}
A generalization of the Gr\"{u}nwald approximation for the left and right fractional derivatives are presented in terms of generating functions. Using this generalization, a second order difference approximation is constructed. A quasi-compact approximation of order 3
is also obtained by using the second order approximation operators with a preconditioner operator. The approximations are first tested for steady state problems and numerically confirmed their theoretical order.

Numerical schemes for these approximations to solve
fractional diffusion equations are devised with proof of stability and convergence.  The theoretical developments are tested and verified for time dependent diffusion problems through an example.

The approach of generating functions might be a useful
tool for constructing difference approximation formulas for fractional derivatives.

\paragraph{Acknowledgements}
This research is supported by Sultan Qaboos University Internal Grant No. IG/SCI/DOMAS/16/10.

\bibliographystyle{acm}
\bibliography{NasirNafabib}

\end{document}